\documentclass[12pt,letterpaper]{amsart}

\usepackage[a4paper,margin=27mm]{geometry}
\usepackage{amsmath,amssymb,amsthm}
\usepackage{enumitem}
\usepackage{graphicx}
\usepackage[colorlinks=true,linkcolor=blue,citecolor=blue,urlcolor=blue]{hyperref}

\allowdisplaybreaks

\newtheorem{theorem}{Theorem}[section]
\newtheorem{lemma}[theorem]{Lemma}
\newtheorem{remark}[theorem]{Remark}
\newtheorem{conjecture}[theorem]{Conjecture}
\newtheorem{proposition}[theorem]{Proposition}

\DeclareMathOperator{\Par}{Par}

\newcommand{\cE}{\mathcal E}
\newcommand{\one}{\mathbf 1}
\usepackage{xcolor}

\newcommand{\BPD}[2][1pc]{%
\setlength{\unitlength}{#1}
\def\BPDframe{%
    \thinlines
    \color{black!70}
    \put(0,0){\line(0,1){1}}
    \put(1,0){\line(0,1){1}}
    \put(0,0){\line(1,0){1}}
    \put(0,1){\line(1,0){1}}}%
\def\O{%
\begin{picture}(1,1)
    \BPDframe
\end{picture}}%
\def\PsiCell##1{%
  \begin{picture}(1,1)
    \put(0,0){\color{teal!45}\rule{\unitlength}{\unitlength}}
    \BPDframe
    \put(0,0.27){\makebox[\unitlength]{\color{black}\scriptsize \(##1\)}}
  \end{picture}}%
\def\CompCell##1{%
  \begin{picture}(1,1)
    \put(0,0){\color{orange!35}\rule{\unitlength}{\unitlength}}
    \BPDframe
    \put(0,0.27){\makebox[\unitlength]{\color{black}\scriptsize \(##1\)}}
  \end{picture}}%
\def\RootFaint{%
  \begin{picture}(1,1)
    \put(0,0){\color{gray!15}\rule{\unitlength}{\unitlength}}
    \BPDframe
  \end{picture}}%
  \def\DeltaCell##1{%
  \begin{picture}(1,1)
    \put(0,0){%
      \color{blue!28}%
      \rule{\unitlength}{\unitlength}}
    \BPDframe
    \put(0,0.27){%
      \makebox[\unitlength]{%
        \color{black}\scriptsize \(##1\)}}
  \end{picture}}
\def\RootFaint{%
  \begin{picture}(1,1)
    \put(0,0){%
      \color{gray!10}%
      \rule{\unitlength}{\unitlength}}
    \BPDframe
  \end{picture}}
  \def\RemCell##1{%
  \begin{picture}(1,1)
    \put(0,0){\color{red!25}\rule{\unitlength}{\unitlength}}
    \BPDframe
    \put(0,0.27){%
      \makebox[\unitlength]{%
        \color{red!70!black}\scriptsize \(##1\)}}
  \end{picture}}%
\begin{array}{@{\,}c@{\,}}{}
{\def\arraystretch{0}
\setlength{\arraycolsep}{0pt}
\begin{array}{@{}l@{}}#2\end{array}}
\end{array}}

\title{ $K$-$k$-Schur functions are  $k$-Schur positive}

\date{}

\author{Haojun Bai}
\address[Haojun Bai]{Center for Combinatorics, LPMC, 
Nankai University, Tianjin 300071, P.R. China}
\email{haojunbai@mail.nankai.edu.cn}

\author{Peter L. Guo}
\address[Peter L. Guo]{Center for Combinatorics, LPMC, 
Nankai University, Tianjin 300071, P.R. China}
\email{lguo@nankai.edu.cn}

\begin{document}

\begin{abstract}
The $k$-Schur functions (resp., $K$-$k$-Schur functions) are the symmetric function  representatives of Schubert classes in the homology (resp., $K$-homology) of the affine Grassmannian of type $A$. 
We prove  that every $K$-$k$-Schur function is a nonnegative integer linear combination in the basis of  $k$-Schur functions. This settles   a conjecture  by Lam, Schilling  and Shimozono. The proof relies  on realizations of $k$-Schur   and $K$-$k$-Schur functions as    Catalan and  Katalan functions, respectively. 
\end{abstract}

\maketitle


\vspace{10pt}


\section{Introduction}

Throughout this paper, let $k$ be a positive integer. The $k$-Schur functions $s^{(k)}_{\lambda}(\mathbf{x})$, which are homogeneous symmetric functions  indexed by $k$-bounded   partitions $\lambda$ (namely, the largest part $\lambda_1\leq k$), were discovered by Lapointe,   Lascoux and   Morse   
\cite{LLM} in their attempt to strengthen the   
Macdonald positivity conjecture.\footnote{The original definition of a $k$-Schur function depends on an additional parameter $t$, and in this work we concentrate on  the $t=1$ case.}  Geometrically, Lam \cite{Lam} established that $k$-Schur functions are the  symmetric function representatives of Schubert classes in the homology
of the affine Grassmannian of type $A$. For combinatorial constructions of 
$s^{(k)}_{\lambda}(\mathbf{x})$, 
see for example \cite{LLMS1,LM,LM1}. We also refer to the monograph
\cite{Survey} for a systematic treatment of $k$-Schur functions.

From a different perspective,   Chen and   Haiman \cite{Chen} conjectured
that $k$-Schur functions are a subclass of a family of symmetric functions, called Catalan functions,  which appear as the Euler
characteristics of  vector bundles on  the flag variety    \cite{Chen,Panyushev}. 
This prediction was confirmed  by  Blasiak, Morse,  Pun  and  Summers \cite[Theorem 2.4]{BMPS}.

There is a parallel story on the $K$-theoretic analogue of $k$-Schur functions.  Lam, Schilling  and
Shimozono \cite{LSS} constructed the symmetric functions   $g^{(k)}_\lambda(\mathbf{x})$, now called $K$-$k$-Schur functions,  and
identified them with Schubert classes in the  $K$-homology of the  affine Grassmannian of type $A$ \cite[Theorem 7.17]{LSS}.  The $K$-$k$-Schur function  $g^{(k)}_\lambda(\mathbf{x})$ is inhomogeneous, and its highest-degree  homogeneous
component is equal to the $k$-Schur function  $s^{(k)}_\lambda(\mathbf{x})$ \cite[Lemma 7.15]{LSS}.
A combinatorial model of $g^{(k)}_\lambda(\mathbf{x})$ based on  affine set-valued tableaux was found by Morse \cite{Morse}. 
As a simultaneous  extension of $k$-Schur  and Catalan functions, 
 Blasiak,  Morse  and  Seelinger \cite{BMS} introduced a family of symmetric functions, called   Katalan functions, with remarkable properties.

Our aim  is to prove that every $K$-$k$-Schur function is  $k$-Schur  positive, resolving a conjecture posed   by 
Lam, Schilling  and Shimozono  \cite[Conjecture 7.20(i)]{LSS}.

\begin{theorem}[\text{\cite[Conjecture 7.20(i)]{LSS}}]\label{thm:LSS}
For every $k$-bounded partition $\lambda$,
\[
 g^{(k)}_\lambda({\mathbf x})=
\sum_{\nu}b^{(k)}_{\lambda,\nu}\, s^{(k)}_\nu(\mathbf x),
 \qquad \text{where $b^{(k)}_{\lambda,\nu}\in\mathbb{Z}_{\ge0}$}.
\]
  
\end{theorem}

\begin{remark}
Taking large $k$ (for example, $k$ is larger than the size of $\lambda$),  $ g^{(k)}_\lambda(\mathbf x)$ equals the dual stable Grothendieck polynomial $ g_{\lambda}(\mathbf x)$ \cite[Corollary 2.9]{BMS}, while   $ s^{(k)}_\lambda(\mathbf x)$ equals  the classical Schur function  $ s_{\lambda}(\mathbf x)$ \cite[Property  39]{LM1}. So Theorem \ref{thm:LSS} implies the Schur positivity of  $ g_\lambda(\mathbf x)$, which was first  established by Lam and   Pylyavskyy \cite[Theorem 9.8]{LP}, see also  Guo, Kang and Liu \cite[Theorem 1.3]{GKL} for a vast generalization.   
\end{remark}

\begin{remark}\label{rek-0-k}
It was conjectured by Blasiak,   Morse  and Seelinger    \cite[Conjecture 2.19]{BMS} that  Theorem \ref{thm:LSS} holds for a broad class  of Katalan functions. In Section \ref{sec-f-9}, we show that there exist  counterexamples to this conjecture. 
\end{remark}

The basic idea in the proof of Theorem \ref{thm:LSS} is to utilize  the realization  $ s^{(k)}_\lambda(\mathbf x)=H(\Delta^k(\lambda);\lambda)$  \cite[Theorem 2.4]{BMPS} as a Catalan function  and the realization  \(
 g^{(k)}_\lambda(\mathbf x)=
 K(\Delta^k(\lambda);\Delta^{k+1}(\lambda);\lambda)
\) \cite[Theorem 2.6]{BMS} as a     Katalan function,  see Section \ref{ssec2} for detailed explanations. To achieve our goal, we need to work in a broader setting 
 \[
\widetilde{\mathrm{Par}}_\ell^k =
 \{\gamma\in\mathbb{Z}^\ell:\gamma_i\le k\text{ and }
       \gamma_{i+1}-\gamma_i\le1\text{ for }i<\ell\},
\] 
which contains $k$-bounded partitions of length at most $\ell$ as a subclass. Actually, we shall prove a stronger $k$-Schur positivity in Theorem \ref{thm:main}, whose proof borrows ideas from the breakthrough work by  Blasiak,  Morse,  Pun and  Summers \cite{BMPS}.

This paper is structured  as follows. In  Section \ref{ssec2}, we review the definitions of Catalan and Katalan functions, as well as their specializations to $k$-Schur and $K$-$k$-Schur functions. In Section \ref{xhn-0-9}, we  give a stronger $k$-Schur positivity  result, see Theorem \ref{thm:main}, based on which we shall finish the proof of 
Theorem \ref{thm:LSS}.  Section \ref{se-0-87} is devoted to a proof of  Theorem \ref{thm:main}. 
In Section \ref{sec-f-9}, we exhibit counterexamples to the conjecture by  Blasiak,   Morse  and Seelinger    \cite[Conjecture 2.19]{BMS}.

\subsection*{Acknowledgments}
This work was supported by the National Natural Science Foundation of China (No. 12371329) and the Fundamental Research Funds for the Central Universities (No. 63263094).

\subsection*{AI disclosure}
Generative AI was used interactively as a technical assistant for checking calculations and proof strategies proposed by the authors. The counterexample given in Section \ref{sec-f-9} was found by ChatGPT-6 Astra.
All resulting material was reviewed and  verified by the authors.

\section{Catalan and Katalan functions}\label{ssec2}

In this section, we give an overview of   Catalan and Katalan functions, and explain how they specialize to $k$-Schur and $K$-$k$-Schur  functions, respectively.

\subsection{Catalan functions and $k$-Schur functions}

Catalan functions appear  naturally as $\mathrm{GL}_\ell(\mathbb C)$-equivariant Euler
characteristics of certain vector bundles on  the flag variety    \cite{Chen,Panyushev}. Here  we adopt the algebraic description from \cite{BMPS}. 
Let
\[
 \Delta^+_\ell=\{(i,j)\colon 1\le i<j\le\ell\}
\]
be the set of labels indexing the  positive roots of type $A_{\ell-1}$. We simply call $(i,j)$ a root of $ \Delta^+_\ell$.  The partial order on $ \Delta^+_\ell$ is defined as follows:  $(i,j)\preceq(p,q)$ when 
$p\le i<j\le q$.  A \emph{root ideal} is an upper order ideal
$\Psi\subseteq\Delta^+_\ell$, which means if $(i,j)\in \Psi$ and $(i,j)\preceq(p,q)$, then  $(p,q)\in \Psi$. Notice that  the complement $\Delta^+_\ell\setminus \Psi$ is a lower order ideal of $\Delta^+_\ell$.

In Figure \ref{fig:root-ideal-example}, we  illustrate   a root ideal $\Psi\subseteq\Delta^+_\ell$ for $\ell=6$ and its complement $\Delta^+_\ell\setminus \Psi$, which are indicated  by  $\BPD[3mm]{\PsiCell{}} $ and $\BPD[3mm]{\CompCell{}}$, respectively.
Here the root $(i,j)$ is put in the box in row $i$ and column $j$ (rows are counted from top to bottom, and columns are counted from left to right). 
\begin{figure}[h t]
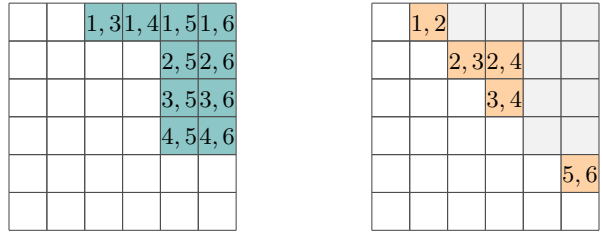

\centering

\[
\begin{array}{c@{\qquad\qquad}c}

\BPD[5mm]{%
  \O\O\PsiCell{1,3}\PsiCell{1,4}\PsiCell{1,5}\PsiCell{1,6}\\
  \O\O\O\O\PsiCell{2,5}\PsiCell{2,6}\\
  \O\O\O\O\PsiCell{3,5}\PsiCell{3,6}\\
  \O\O\O\O\PsiCell{4,5}\PsiCell{4,6}\\
  \O\O\O\O\O\O\\
  \O\O\O\O\O\O
}
&
\BPD[5mm]{%
  \O\CompCell{1,2}\RootFaint\RootFaint\RootFaint\RootFaint\\
  \O\O\CompCell{2,3}\CompCell{2,4}\RootFaint\RootFaint\\
  \O\O\O\CompCell{3,4}\RootFaint\RootFaint\\
  \O\O\O\O\RootFaint\RootFaint\\
  \O\O\O\O\O\CompCell{5,6}\\
  \O\O\O\O\O\O
}
\end{array}
\]

\caption{A root ideal $\Psi$ in $\Delta_6^+$ and its complement
$\Delta_6^+\setminus\Psi$.}
\label{fig:root-ideal-example}
\end{figure}

For $d\geq 1$, let 
\[
h_d(\mathbf{x})=\sum_{i_1\leq \cdots\leq i_d}x_{i_1}\cdots x_{i_d}.
\]
We   use the convention that $h_0(\mathbf{x})=1$ and $h_d(\mathbf{x})=0$ for $d<0$. For an integer vector
$\eta\in\mathbb{Z}^\ell$, write \[h_\eta(\mathbf{x})=\prod_i h_{\eta_i}(\mathbf{x}).\] 
For a root ideal $\Psi\subseteq\Delta^+_\ell$ and a vector $\eta\in\mathbb{Z}^\ell$, the associated  Catalan function   can be defined by
\begin{equation}\label{Cat-1}
 H(\Psi;\eta)=
 \prod_{(i,j)\in\Delta^+_\ell\setminus\Psi}(1-R_{ij})\,h_\eta,
\end{equation}
see \cite[Corollary 4.8]{BMPS}.
Here $ R_{ij}$ is the raising operator acting on the subscripts   by
\[
R_{ij}\, \eta ={\eta+\epsilon_i-\epsilon_j},
\]
where $\epsilon_i\in \mathbb{Z}^\ell$ is the $i$-th standard coordinate vector.

Catalan functions contain $k$-Schur functions as special cases. 
Recall that a  {partition of length at most $\ell$} is a weakly decreasing vector
$\lambda=(\lambda_1,\ldots,\lambda_\ell)$ of nonnegative integers. We say that $\lambda$  is  {$k$-bounded} if $\lambda_1\le k$.
Write $\Par^k$ for the set
of $k$-bounded partitions, and  $\Par^k_\ell$ for the set of $k$-bounded partitions of length at most $\ell$. 
The
set $\{s^{(k)}_\lambda(\mathbf{x}):\lambda\in\Par^k\}$ of $k$-Schur functions  forms a basis of
the subalgebra
\[
 \Lambda_{(k)}=\mathbb{Z}[h_1,\ldots,h_k]
\]
 of symmetric functions.

For $\lambda\in {\mathrm{Par}}_\ell^k$, define a root ideal as follows \cite[Definition 2.2]{BMPS}:
\begin{equation}\label{ideal=09}
 \Delta^k(\lambda)=
 \{(i,j)\in\Delta^+_\ell\colon k-\lambda_i+i<j\}.
\end{equation}  
See Remark \ref{rek-0-8} for a concrete description of $ \Delta^k(\lambda)$.
Figure \ref{fig:Delta-k-example} depicts  the root ideal  $\Delta^k(\lambda)$ for $k=4$, $\ell=6$ and 
$\lambda=(4,3,2,2,1,0)$. 
\begin{figure}[h t]
\centering
\[
\BPD[5mm]{%
  \O\DeltaCell{1,2}\DeltaCell{1,3}\DeltaCell{1,4}\DeltaCell{1,5}\DeltaCell{1,6}\\
  \O\O\RootFaint\DeltaCell{2,4}\DeltaCell{2,5}\DeltaCell{2,6}\\
  \O\O\O\RootFaint\RootFaint\DeltaCell{3,6}\\
  \O\O\O\O\RootFaint\RootFaint\\
  \O\O\O\O\O\RootFaint\\
  \O\O\O\O\O\O
}
\]

\caption{The root ideal $\Delta^4(\lambda)$ for
$\lambda=(4,3,2,2,1,0)$.}
\label{fig:Delta-k-example}
\end{figure}

It was conjectured by Chen and   Haiman \cite{Chen} and  proved by Blasiak, Morse,  Pun  and  Summers  \cite[Theorem 2.4]{BMPS} that  for $\lambda\in\Par^k_\ell$,
\begin{equation}\label{cdj-0-kl}
s^{(k)}_\lambda(\mathbf{x})=H(\Delta^k(\lambda);\lambda).
\end{equation}

\subsection{Katalan functions and $K$-$k$-Schur functions}

We proceed to review the Katalan functions introduced in \cite{BMS}. 
For $\eta=(\eta_1,\ldots,\eta_\ell)\in\mathbb{Z}^\ell$, set
\[
 k_\eta
 =
 k_{\eta_1}^{(0)}\,k_{\eta_2}^{(1)}\cdots
 k_{\eta_\ell}^{(\ell-1)},
\]
where $k_m^{(r)}=0$ for $m<0$, while for $m\ge0$,
\[
 k_m^{(r)}
 =
 \begin{cases}
  h_m,&r=0,\\[5pt]
  \displaystyle\sum_{d=0}^{m}\binom{r+d-1}{d}h_{m-d},&r\ge1.
 \end{cases}
\]  
Given  a root ideal $\Psi\subseteq\Delta^+_\ell$ and   a finite multiset $M$ on $\{1,\ldots,\ell\}$, the Katalan function can be  defined by (see \cite[Proposition 2.2]{BMS}):
\begin{equation}\label{xbb-0}
 K(\Psi;M;\eta)=
 \prod_{j\in M}(1-L_j)
       \prod_{(i,j)\in\Delta^+_\ell\setminus\Psi}(1-R_{ij})\,k_\eta.
\end{equation}
Here the raising operator  $R_{ij}$ is  as defined above, and $L_j$  is the  lowering  operator  acting  on the subscripts by
\[
 \qquad L_j \, \eta={\eta-\epsilon_j}.
\] 

The multisets we encounter often come from root ideals of $\Delta^+_\ell$. 
 For any root ideal \(\Phi\subseteq\Delta^+_\ell\), write 
\begin{equation}\label{xb-9-xn}
M(\Phi)=\bigsqcup_{(i,j)\in\Phi}\{j\}. 
\end{equation}
For  root ideals $\Psi, \Phi\subseteq\Delta^+_\ell$ and $\eta\in \mathbb{Z}^\ell$, 
set
\[
K(\Psi;\Phi;\eta):=K(\Psi;M(\Phi);\eta).
\]
The Catalan and Katalan functions are connected via the following relation \cite[Proposition 2.3]{BMS}:
\[
H(\Psi; \eta)=K(\Psi;\Delta_\ell^+;\eta).
\]
Moreover, 
Blasiak, Morse  and Seelinger \cite[Theorem 2.6]{BMS}  showed that  for $\lambda\in\Par^k_\ell$, 
\begin{equation*}\label{eq:BMSintro}
 g^{(k)}_\lambda(\mathbf{x})=
 K(\Delta^k(\lambda);\Delta^{k+1}(\lambda);\lambda).
\end{equation*}

\section{Proof of Theorem \ref{thm:LSS}}\label{xhn-0-9}

In this section, we shall  state a stronger result than 
Theorem \ref{thm:LSS}, see Theorem \ref{thm:main},   whose proof will be given in Section \ref{se-0-87}. 
As a direct consequence of Theorem \ref{thm:main}, we arrive at  a  proof of Theorem \ref{thm:LSS}.

To describe Theorem \ref{thm:main}, we work in the following broader  setting 
than  $\Par^k_\ell$:
 \[
\widetilde{\mathrm{Par}}_\ell^k =
 \{\gamma\in\mathbb{Z}^\ell\colon \gamma_i\le k\text{ and }
       \gamma_{i+1}-\gamma_i\le1\text{ for }i<\ell\}.
\] 
which has been investigated  in \cite[Section 7]{BMPS}. 
Clearly,  $\Par^k_\ell\subseteq \widetilde{\mathrm{Par}}_\ell^k$. 
In the same way as in \eqref{ideal=09},
for every $\gamma\in\widetilde{\mathrm{Par}}_\ell^k$, we may define a root ideal \cite[Definition 7.1]{BMPS}:
\[
 \Delta^k(\gamma)=
 \{(i,j)\in\Delta^+_\ell\colon k-\gamma_i+i<j\}.
\]

The following remark will be frequently used in the remainder.

\begin{remark}\label{rek-0-8}
For  $\gamma\in\widetilde{\mathrm{Par}}_\ell^k$, there is a concrete description (essentially a restatement) for the root ideal $ \Delta^k(\gamma)$ \cite[Proposition 7.3]{BMPS}.    Specifically,
set
\begin{equation*}\label{eq:rho}
\rho_i(\gamma)=k+1-\gamma_i+i,
 \qquad \rho_{\ell+1}=+\infty.
\end{equation*}
Notice that for $1\leq i\leq \ell$, 
we have $\rho_i(\gamma)>i$ and 
\[
 \rho_{i+1}(\gamma)-\rho_i(\gamma)=1-(\gamma_{i+1}-\gamma_i)\ge0.
\]
Then, for $1\leq i\leq \ell$, the $i$-th row of $\Delta^k(\gamma)$ begins with the root $(i, \rho_i(\gamma))$ if $\rho_i(\gamma)\le\ell$. 
\end{remark}
 
In  Figure  \ref{fig:Delta-k-nonpartition}, we  illustrate the root ideal $\Delta^k(\gamma)$ with $k=4$, $\ell=6$ and $\gamma=(3,4,2,3,1,0)$.
Notice that 
\[
(\rho_1(\gamma),\ldots, \rho_\ell(\gamma))=(3,3,6,6,9,11).
\]
\begin{figure}[h t]
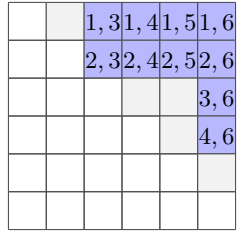

\centering

\[
\BPD[5mm]{%
  \O\RootFaint
    \DeltaCell{1,3}\DeltaCell{1,4}
    \DeltaCell{1,5}\DeltaCell{1,6}\\
  \O\O
    \DeltaCell{2,3}\DeltaCell{2,4}
    \DeltaCell{2,5}\DeltaCell{2,6}\\
  \O\O\O
    \RootFaint\RootFaint\DeltaCell{3,6}\\
  \O\O\O\O
    \RootFaint\DeltaCell{4,6}\\
  \O\O\O\O\O
    \RootFaint\\
  \O\O\O\O\O\O
}
\]

\caption{The root ideal $\Delta^4(\gamma)$ for $\gamma=(3,4,2,3,1,0)\in
\widetilde{\mathrm{Par}}_6^4$. }
\label{fig:Delta-k-nonpartition}
\end{figure}

For
$\gamma\in\widetilde{\mathrm{Par}}_\ell^k$ and
$c=(c_1,\ldots,c_\ell)\in\mathbb{Z}_{\ge0}^\ell$, define
\begin{align}\label{eq:Fintro}
 F_{\gamma,c}(u)&=
 \prod_{j=1}^{\ell}(1-uL_j)^{-c_j}
 \,H(\Delta^k(\gamma);\gamma)\nonumber\\
&= \prod_{j=1}^{\ell}(1-uL_j)^{-c_j}\prod_{(i,j)\in\Delta^+_\ell\setminus\Delta^k(\gamma)}(1-R_{ij})\,h_\gamma\in \Lambda[[u]],
\end{align}
where 
\[
(1-uL_j)^{-c_j}=\sum_{d\geq 0} {c_j+d-1 \choose d} L_j^d\, u^d.
\]
Note that since $L_j$ decreases the subscripts strictly, $F_{\gamma,c}(u)$  actually belongs to $\Lambda[u]$. 

Our main result is  the following   $k$-Schur positivity  statement.

\begin{theorem}\label{thm:main}
Let $\gamma\in\widetilde{\mathrm{Par}}_\ell^k$, and let
$c=(c_1\le\cdots\le c_\ell)\in \mathbb{Z}_{\geq 0}^\ell$ be weakly increasing.  For every $r\ge0$,
\[
 [u^r]F_{\gamma,c}(u)
 \in\sum_{\nu\in\Par^k_\ell}\mathbb{Z}_{\ge0}\cdot s^{(k)}_\nu.
\]
\end{theorem}

Indeed, Theorem   \ref{thm:main} is stronger than  Theorem \ref{thm:LSS}.  In the remainder of this section, we show that Theorem \ref{thm:main} implies  Theorem \ref{thm:LSS}.  To this end, we introduce  two linear operators $\mathbb H$ and $\mathbb K$ from the ring $\mathbb{Z}[z_1^{\pm1},\ldots,z_\ell^{\pm1}]$  of Laurent polynomials in $z_1,\ldots, z_\ell$ to the ring $\Lambda$ of symmetric functions. Precisely, the linear operator $\mathbb H$ (resp., $\mathbb K$) is determined by $\mathbb H(z^\eta)=h_\eta$ (resp., $\mathbb K(z^\eta)=k_\eta$) for $\eta\in \mathbb{Z}^\ell$. 
Noticing  that 
\begin{align*}
k_m^{(r)}=\sum_{d=0}^{m}\binom{r+d-1}{d}h_{m-d}
=\sum_{d=0}^{m}\binom{r+d-1}{d} L_1^d\,h_{m}=(1-L_1)^{-r}\,h_m, 
\end{align*}
these two operators are connected via the following relation:
\begin{equation}\label{agn-0-1}
    \mathbb K(z^\eta)=k_{\eta_1}^{(0)}\,k_{\eta_2}^{(1)}\cdots
 k_{\eta_\ell}^{(\ell-1)}=\mathbb H\left(\prod_{j=1}^\ell \left(1-\frac{1}{z_j}\right)^{-(j-1)} z^\eta\right).
\end{equation}

\begin{remark}
 On the right-hand side of   \eqref{agn-0-1}, 
 \begin{equation}\label{sti-rmk}
 \prod_{j=1}^\ell \left(1-\frac{1}{z_j}\right)^{-(j-1)} z^\eta
 \end{equation}
 is generally  not a Laurent polynomial. However, since $\mathbb H(z^\theta)=0$ for $\theta\not\in \mathbb{Z}_{\geq 0}^\ell$,  there are only  finitely many Laurent monomials in \eqref{sti-rmk} that can contribute. We will not elaborate further on this when such situations arise later.
\end{remark}

According to  the definitions in \eqref{Cat-1} and \eqref{xbb-0}, it is immediate  that 
\begin{equation}\label{Cat-1-1}
 H(\Psi;\eta)=\mathbb{H}\left(
\prod_{(i,j)\in\Delta^+_\ell\setminus\Psi} \left(1-\frac{z_i}{z_j}\right) z^\eta\right)
\end{equation}
and 
\begin{equation}\label{xbb-0-1}
 K(\Psi;M;\eta)=
\mathbb{K}\left( \prod_{j\in M}\left(1-\frac{1}{z_j}\right)
       \prod_{(i,j)\in\Delta^+_\ell\setminus\Psi}\left(1-\frac{z_i}{z_j}\right)\,z^\eta\right).
\end{equation}

We now deduce Theorem \ref{thm:LSS} from  Theorem   \ref{thm:main}.

\begin{proof}[Proof of  Theorem \ref{thm:LSS}.]
 Let $\lambda\in \mathrm{Par}_\ell^k$.
Suppose that  $m_j$ is the number of roots of $\Delta^{k+1}(\lambda)$ in column $j$. 
Using the notation in \eqref{xb-9-xn},
\[
M(\Delta^{k+1}(\lambda))=\{1^{m_1},2^{m_2},\ldots, \ell^{m_\ell}\}.
\]
Set 
\[
c_j=j-1-m_j.
\]
By Remark \ref{rek-0-8}, we see that 
\[
 c_j=\#\{i<j\colon k+1-\lambda_i+i\geq j\}.
\]
Clearly $c_1=0$.  Moreover, for $1\leq j<\ell$,
\begin{equation}\label{nbj=0-9}
 c_{j+1}-c_j
 =1-\#\{i<j\colon k+1-\lambda_i+i=j\}.
\end{equation}
Since the values   $\lambda_i-i$ are strictly decreasing,   the set on the right-hand side of  \eqref{nbj=0-9} has at most one element. So we have \[0\le c_1\le\cdots\le c_\ell.\]

Combining \eqref{Cat-1-1} and \eqref{xbb-0-1} and together with \eqref{agn-0-1}, we deduce that  
\begin{align*}
 g^{(k)}_\lambda&=K(\Delta^k(\lambda);\Delta^{k+1}(\lambda);\lambda)\\
 &=\mathbb K\left(\prod_{j=1}^\ell\left(1-\frac{1}{z_j}\right)^{m_j}
       \prod_{(i,j)\in\Delta^+_\ell\setminus\Delta^k(\lambda)}\left(1-\frac{z_i}{z_j}\right)\, z^\lambda\right)\\
  &=\mathbb H\left(\prod_{j=1}^\ell\left(1-\frac{1}{z_j}\right)^{-(j-1)} \prod_{j=1}^\ell\left(1-\frac{1}{z_j}\right)^{m_j}
       \prod_{(i,j)\in\Delta^+_\ell\setminus\Delta^k(\lambda)}\left(1-\frac{z_i}{z_j}\right)  z^\lambda\right)\\  
       &=\mathbb H\left(\prod_{j=1}^\ell\left(1-\frac{1}{z_j}\right)^{-c_j}
       \prod_{(i,j)\in\Delta^+_\ell\setminus\Delta^k(\lambda)}\left(1-\frac{z_i}{z_j}\right) z^\lambda\right)\\ 
&= \prod_{j=1}^{\ell}(1-L_j)^{-c_j}
 H(\Delta^k(\lambda);\lambda)=F_{\lambda,c}(1).
\end{align*}
Invoking  Theorem \ref{thm:main}, we conclude  that $g^{(k)}_\lambda$ is $k$-Schur positive.
\end{proof}

\section{Proof of Theorem \ref{thm:main}}\label{se-0-87}

In this section, we provide a proof of Theorem \ref{thm:main}. 
Let us  start with some notation that will be used later. 
For a root ideal $\Psi\subseteq\Delta^+_{\ell}$, denote 
\[
P_\Psi(z)=\prod_{(i,j)\in\Delta^+_{\ell}\setminus\Psi}
          \left(1-\frac{z_i}{z_j}\right).
\]
Recall  from  \eqref{Cat-1-1} that  for $\eta\in \mathbb{Z}^\ell$,
\[ H(\Psi;\eta)=\mathbb H \bigl(P_\Psi(z)z^\eta\bigr).
\]

Suppose that 
\[
Q(u;z)\in \mathbb{Z}[z_1^{\pm1},\ldots,z_\ell^{\pm1}][[u]]
\]
is a formal power series in $u$ whose coefficients are Laurent polynomials in $z_1,\ldots, z_\ell$.  Define
\begin{equation}\label{def-0-9}
 \cE_Q(\Psi;\eta)=\mathbb H \bigl(Q(u;z)P_\Psi(z)z^\eta\bigr)\in \Lambda[[u]].
\end{equation} 
In this notation, for $\gamma\in\widetilde{\mathrm{Par}}_\ell^k$ and
$c=(c_1,\ldots,c_\ell)\in \mathbb{Z}_{\geq 0}^\ell$, \eqref{eq:Fintro} can be rewritten as 
\[
 F_{\gamma,c}(u)=\cE_{Q_c}(\Delta^k(\gamma);\gamma),
\]
where 
\begin{equation}\label{qc-k-l}
Q_c(u;z)=\prod_{j=1}^{\ell}\left(1-\frac{u}{z_j}\right)^{-c_j}.
\end{equation}

\subsection{Properties  concerning $\cE_Q(\Delta^k(\gamma);\gamma)$ for $\gamma\in\widetilde{\mathrm{Par}}_\ell^k$}

Throughout this subsection, we always assume that  $\Psi\subseteq\Delta^+_{\ell}$ is a root ideal, and that 
\[Q(u;z)\in \mathbb{Z}[z_1^{\pm1},\ldots,z_\ell^{\pm1}][[u]].\]
We shall prove several   properties of $\cE_Q(\Delta^k(\gamma);\gamma)$ for $\gamma\in\widetilde{\mathrm{Par}}_\ell^k$, which are largely inspired by results from   \cite{BMPS,BMS}.

Write $[\ell]=\{1,\ldots, \ell\}$. 
For $1\leq i\leq \ell-1$, denote  by $\tau_i$   the simple reflection on $[\ell]$ exchanging $i$ and $i+1$, which acts on a Laurent  monomial  $z^\eta$ ($\eta\in \mathbb Z^\ell$) by exchanging   $z_i$ and $z_{i+1}$.  Moreover, for   $(a,b)\in [\ell]\times [\ell]$, define  $\tau_i (a,b)=(\tau_i(a), \tau_i(b))$. For a subset $\Theta\subseteq[\ell]\times [\ell]$, let 
\[\tau_i\, \Theta=\{\tau_i (a,b)\colon (a,b)\in \Theta\}\subseteq[\ell]\times [\ell].\]

The first  lemma can be regarded as an Laurent extension of \cite[Lemma 6.2]{BMPS}. 

\begin{lemma}\label{lem:mirror}
Let $1\leq i\leq \ell-1$.  Suppose that both  $\Psi$ and $Q(u;z)$  are invariant under  the action of $\tau_i$. For $\eta\in \mathbb Z^\ell$ with $\eta_i=\eta_{i+1}-1$, we have 
\[
 \cE_Q(\Psi;\eta)=0.
\]
\end{lemma}

\begin{proof}
Since  $\tau_i\,\Psi=\Psi$, we see that $(i,i+1)\notin\Psi$ and
\[
 P_\Psi(z)=\left(1-\frac{z_i}{z_{i+1}}\right)C(z),
\]
where $C(z)$ is invariant under $\tau_i$. 
Along  with the condition $\eta_i=\eta_{i+1}-1$, it can be directly checked that
\[
 \tau_i \left(QP_\Psi z^\eta\right)
 =-QP_\Psi z^\eta.
\]
Noticing that  $\mathbb H\circ\tau_i=\mathbb H$, we deduce  $ \cE_Q(\Psi;\eta)=- \cE_Q(\Psi;\eta)$, and thus $ \cE_Q(\Psi;\eta)=0$.  
\end{proof}

A root $\alpha\in\Psi$ is \emph{removable} if $\Psi\setminus\{\alpha\}$ is   a root ideal, and a root $\alpha\notin\Psi$ is \emph{addable} if $\Psi\cup\{\alpha\}$ is a root ideal.  Using  Remark \ref{rek-0-8}, for $\gamma\in\widetilde{\mathrm{Par}}_\ell^k$, a root $(a,b)\in \Delta^k(\gamma)$ is removable  precisely when
\[
 b=\rho_a(\gamma)\le\ell,\qquad \rho_a(\gamma)<\rho_{a+1}(\gamma).
\]


Using  Lemma \ref{lem:mirror}, we can prove the following proposition  which  is a   Laurent version  of the local \(k\)-Schur straightening rule in \cite[Lemma 7.4]{BMPS}.

\begin{proposition}\label{prop:straight}
Let $\gamma\in\widetilde{\mathrm{Par}}_\ell^k$,  and let $2\le b\le\ell$ be such that $ \gamma_{b-1}=\gamma_b-1$ and $\gamma_b\ge\gamma_{b+1}$ (where we set $\gamma_{\ell+1}=-\infty$).   Suppose that $Q$ is invariant under $\tau_{b-1}$. 
If there exists a root $(a,b)$ that is removable  from  $\Delta^k(\gamma)$   and
\(\gamma_{a-1}\ge\gamma_a 
\) (where we set $\gamma_{0}=+\infty$), 
then  $\mu=\gamma+\epsilon_a-\epsilon_b\in \widetilde{\mathrm{Par}}_\ell^k$ and 
\begin{equation}\label{equ-m-0}
\cE_Q(\Delta^k(\gamma);\gamma)=\cE_Q(\Delta^k(\mu);\mu).
\end{equation}
If there is no removable root in column $b$, then 
\begin{equation}\label{cjn-eq-0}
\cE_Q(\Delta^k(\gamma);\gamma)=0.
\end{equation}
\end{proposition}

\begin{proof}
Assume first that there exists a removable root $(a,b)$  of $\Delta^k(\gamma)$ and \(\gamma_{a-1}\ge\gamma_a
\). Note that $a\neq b-1$, since otherwise we would have   $\rho_a(\gamma)=\rho_b(\gamma)>b$ from $ \gamma_{b-1}=\gamma_b-1$, contrary to the assumption that $(a,b)$ is removable. Now we have $a\leq b-2$. To show  $\mu\in\widetilde{\mathrm{Par}}_\ell^k$, it suffices to check  that  \(\mu_a\leq k,\ \mu_{a}- \mu_{a-1}\leq 1 \text{  and } \mu_{b+1}- \mu_{b}\leq 1\), which follows respectively  from $a+2\leq b=\rho_a(\gamma)=k+1-\gamma_a+a$, \(\gamma_{a-1}\ge\gamma_a
\) and $\gamma_b\ge\gamma_{b+1}$. 
This verifies  $\mu\in\widetilde{\mathrm{Par}}_\ell^k$. 

For simplicity, write 
\[
 \delta=(a,b),\ \alpha=(a,b-1),\ d=\rho_b(\gamma).
\]
The assumption \(\gamma_{a-1}\ge\gamma_a
\) implies  that $\rho_{a-1}(\gamma)\leq \rho_{a}(\gamma)-1=b-1$, and so  $(a-1,b-1)\in \Delta^k(\gamma)$.  Together with $(a,b)\in \Delta^k(\gamma)$, we obtain that 
the root $\alpha$ is addable to $\Delta^k(\gamma)$. Since $\gamma_b\ge\gamma_{b+1}$, we have 
$\rho_{b}(\gamma)<\rho_{b+1}(\gamma)$, and so if $d\le\ell$, then the root $\beta=(b,d)\in \Delta^k(\gamma)$ is removable.  Therefore, it follows from   Remark \ref{rek-0-8} that
\begin{equation*}
\Delta^k(\mu)=
 \begin{cases}
  (\Delta^k(\gamma)\cup\{\alpha\})\setminus\{\beta\},&d\le\ell,\\[5pt]
  \Delta^k(\gamma)\cup\{\alpha\},&d>\ell.
 \end{cases}
\end{equation*}
Now, it is routine to check that 
\begin{align}
&\cE_Q(\Delta^k(\gamma);\gamma)-\cE_Q(\Delta^k(\mu);\mu)\notag\\
 &\quad=
 \cE_Q(\Delta^k(\gamma)\setminus\{\delta\};\gamma)
 -\cE_Q(\Delta^k(\gamma)\cup\{\alpha\};
       \mu+\epsilon_a-\epsilon_{b-1}) \label{eq:threeerrors}\\
 &\qquad\quad+
 \one_{d\le\ell}\cdot 
 \cE_Q(\Delta^k(\gamma)\cup\{\alpha\};
       \mu+\epsilon_b-\epsilon_d),\notag
\end{align}
where $\one_{d\le\ell}$ is the indicator function with value  $1$ when $d\le\ell$, and $0$ otherwise.

We apply  Lemma \ref{lem:mirror} to show that every term on the right-hand side of \eqref{eq:threeerrors} vanishes. To do this, we have the following analysis.
\begin{itemize}
    \item[(1)] First, the assumption $ \gamma_{b-1}=\gamma_b-1$ implies that $\rho_b(\gamma)=\rho_{b-1}(\gamma)$. Second, notice  that   $(i,b-1)\in \Delta^k(\gamma)$ iff $1\leq i\leq a-1$ because  $(a-1,b-1)\in \Delta^k(\gamma)$ and  $\alpha=(a,b-1)\notin \Delta^k(\gamma)$, and that
$(i,b) \in \Delta^k(\gamma) $ iff $1\leq i\leq a$ since $\delta=(a,b)$ is removable.
These two observations (which are in fact the `ceiling' and `wall' conditions given in \cite[Definition 5.4]{BMPS}) imply that  both of the root ideals $\Delta^k(\gamma)\setminus\{\delta\}$ and
$\Delta^k(\gamma)\cup\{\alpha\}$ are invariant under $\tau_{b-1}$. 

    \item[(2)] The   entries at the positions $b-1,b$ 
of $\gamma$, $\mu+\epsilon_a-\epsilon_{b-1}$ and $\mu+\epsilon_b-\epsilon_d$    are 
\(
 (\gamma_{b-1},\gamma_{b-1}+1),\ (\gamma_{b-1}-1,\gamma_{b-1}),\ (\gamma_{b-1},\gamma_{b-1}+1)
\), respectively.
\end{itemize}
Therefore, each term satisfies the conditions required in   Lemma \ref{lem:mirror}, and thus  is equal to zero. This proves 
\eqref{equ-m-0}. 

It remains to consider the case when  there is no $a$ such that  $(a,b)$ is  removable from $\Delta^{k}(\gamma)$. In this case,   note that   $(i,b) \in \Delta^k(\gamma) $ iff $(i,b-1) \in \Delta^k(\gamma) $, and that $ \rho_{b-1}=\rho_b$ since  $ \gamma_{b-1}=\gamma_b-1$. This ensures that $\tau_{b-1}\,\Delta^k(\gamma)=\Delta^k(\gamma)$. Again, we may apply  Lemma \ref{lem:mirror} to yield \eqref{cjn-eq-0}.  So the proof is complete. 
\end{proof}

By drawing  a directed edge $a\to b$ for each removable root $(a,b)$ of $\Psi$, we get  a graph on $[\ell]$ which is called the bounce graph of $\Psi$ \cite[Definition 5.2]{BMPS}.  Clearly,  the bounce graph   is a disjoint union of directed paths,  called bounce paths \cite[Definition 5.2]{BMPS}. Now, consider the  bounce graph of the  root ideal $\Delta^k(\gamma)$ for $\gamma\in\widetilde{\mathrm{Par}}_\ell^k$. 
For $b\in [\ell]$, let $\mathrm{Path}_\gamma(b)$ be the longest directed path starting at $b$ (which is written in the notation  `downpath' in \cite[Page 938]{BMPS}).  We call the path $\mathrm{Path}_\gamma(b)$  \emph{safe} if  $\gamma_v\ge\gamma_{v+1}$ for every vertex $v\in\mathrm{Path}_\gamma(b)$ with   $v<\ell$.

The last lemma borrows   ideas from  \cite[Proposition 5.6, Corollary 5.7]{BMPS}.

\begin{lemma}\label{lem:path}
Let $\gamma\in\widetilde{\mathrm{Par}}_\ell^k$.  If the path $\mathrm{Path}_\gamma(b)$ is safe, then $\gamma-\epsilon_v\in\widetilde{\mathrm{Par}}_\ell^k$ for every vertex $v\in \mathrm{Path}_\gamma(b)$. 
Moreover, 
\begin{equation}\label{eq:path}
 \cE_{Qz_b^{-1}}(\Delta^k(\gamma);\gamma)=
 \sum_{v\in\mathrm{Path}_\gamma(b)}
 \cE_Q(\Delta^k(\gamma-\epsilon_v);\gamma-\epsilon_v).
\end{equation}
\end{lemma}

\begin{proof}
The assertion that $\gamma-\epsilon_v\in\widetilde{\mathrm{Par}}_\ell^k$ is explained as follows. 
Lowering the $v$-th entry preserves all defining inequalities of $\widetilde{\mathrm{Par}}_\ell^k$ except possibly the one between $v$ and $v+1$; that inequality is also true since $\gamma_v\ge\gamma_{v+1}$.

We next verify the relation in  \eqref{eq:path}. 
If $v\to d$ is an edge, then, by Remark \ref{rek-0-8}, it is easily checked that $\Delta^k(\gamma-\epsilon_v)$ is obtained from $\Delta^k(\gamma)$ by deleting the removable root $(v,d)$, namely, 
\[
 \Delta^k(\gamma-\epsilon_v)=\Delta^k(\gamma)\setminus\{(v,d)\}.
\]
By the definition in \eqref{def-0-9}, it is readily checked that 
\begin{equation}\label{Tie-12}
 \cE_{Qz_v^{-1}}(\Delta^k(\gamma);\gamma)=
 \cE_Q(\Delta^k(\gamma-\epsilon_v);\gamma-\epsilon_v)
 +\cE_{Qz_d^{-1}}(\Delta^k(\gamma);\gamma).
\end{equation}
If $v$ is the last vertex of $\mathrm{Path}_\gamma(b)$, then the $v$-th row of $\Delta^k(\gamma)$ has no roots, and so $\Delta^k(\gamma-\epsilon_v)=\Delta^k(\gamma)$, implying that 
\begin{equation}\label{Tie-12-1}
 \cE_{Qz_v^{-1}}(\Delta^k(\gamma);\gamma)=
 \cE_Q(\Delta^k(\gamma-\epsilon_v);\gamma-\epsilon_v).
\end{equation}
Iterating   \eqref{Tie-12} from the first edge of  $\mathrm{Path}_\gamma(b)$ and combining with   \eqref{Tie-12-1},
we eventually arrive at \eqref{eq:path}, as desired.
\end{proof}

\subsection{Proof of   Theorem \ref{thm:main}}

We are now ready to give a proof of  Theorem \ref{thm:main}. 

\begin{proof}[Proof of Theorem~\ref{thm:main}]
 For simplicity, write
\[
 F^{[r]}_{\gamma,c}=[u^r]F_{\gamma,c}(u),
 \qquad |c|=c_1+\cdots+c_\ell.
\]
Keep in mind that 
\[
 F_{\gamma,c}(u)=\cE_{Q_c}(\Delta^k(\gamma);\gamma).
\]

If $\gamma_\ell<0$, then $F_{\gamma,c}(u)=0$.  Indeed, neither a raising operator nor a lowering operator can increase the last coordinate.

Assume now that  $\gamma_\ell\ge0$.  Since $\gamma\in\widetilde{\mathrm{Par}}_\ell^k$,
we have $\gamma_i\geq \gamma_\ell-(\ell-i)\ge-(\ell-i)$ for $i\in [\ell]$.  Hence
\[
 N(\gamma)=\sum_{i=1}^{\ell}i(\gamma_i+\ell-i)
\]
is a nonnegative integer.  We next complete the proof by using  lexicographic induction on
\(
 (r,|c|,N(\gamma)).
\)

First, consider $r=0$.  It follows from   \eqref{eq:Fintro} that \[F^{[0]}_{\gamma,c}=F_{\gamma,c}(0)=H(\Delta^k(\gamma);\gamma).\]  If $\gamma$ is a partition, then, according to the identification in  \eqref{cdj-0-kl}, we have $H(\Delta^k(\gamma);\gamma)=s^{(k)}_\gamma$. We next consider the case when $\gamma$ is not a partition. We need the following crucial observation. 

\begin{itemize}
    \item[Claim.] 
Let $\gamma\in\widetilde{\mathrm{Par}}_\ell^k$ with $\gamma_\ell\ge0$.  If $\gamma$ is not a partition, then there exists  an index $b\in\{2,\ldots,\ell\}$ such that
\begin{itemize}
 \item[(i)] $\gamma_{b-1}=\gamma_b-1$;
 
 \item[(ii)] The path $\mathrm{Path}_\gamma(b)$ in the bounce graph of  
 $\Delta^k(\gamma)$ is safe;
 
 \item[(iii)] If there exists an edge $a\to b$ where $ a>1$, then  $\gamma_{a-1}\ge\gamma_a$.
\end{itemize}

\end{itemize}

 \begin{proof}[Proof of Claim.]
Let $p$ be the rightmost ascent of $\gamma$, namely, the rightmost index such that  $\gamma_p<\gamma_{p+1}$. Initially set $b=p+1$.  Since there are no ascents at or to the right of $b$, the path $\mathrm{Path}_\gamma(b)$ is safe.

 Suppose that $a\to b$ is an edge in the bounce graph of  
 $\Delta^k(\gamma)$, where 
 $ a>1$. If $\gamma_{a-1}\not\ge\gamma_a$, then 
$\gamma_{a-1}=\gamma_a-1$. In this case, replace $b$ by $a$. Since $(a,b)$ is removable, by Remark \ref{rek-0-8}, we must have   $\gamma_a\ge\gamma_{a+1}$. Note that
\[
\mathrm{Path}_\gamma(a)=\{a\}\sqcup\mathrm{Path}_\gamma(b).
\]
Thus safety is preserved.  Each replacement strictly decreases $b$, so the procedure stops.  At the stopping index, all three requirements in (i), (ii) and (iii) are satisfied. 
\end{proof}

Choose an ascent $b$ as in the above Claim. Recall the power series $Q_c(u;z)$  in \eqref{qc-k-l}. When $u=0$,    $Q_c$ is identically equal to $1$.  Note that $\gamma$, $b$, $a$ and $Q_c$ satisfy all the conditions in 
Proposition \ref{prop:straight}. Therefore, $F_{\gamma,c}(0)=\cE_{Q_c}(\Delta^k(\gamma);\gamma)$ is either  zero or equal to $\cE_{Q_c}(\Delta^k(\mu);\mu)$ where
$\mu=\gamma+\epsilon_a-\epsilon_b$.  In the latter case, we compute that 
\[
 N(\mu)=N(\gamma)-(b-a)<N(\gamma).
\]
Since $(r, |c|, N(\mu))$ is less than $(r, |c|, N(\gamma))$, we may use induction to conclude that $F_{\gamma,c}(0)$  is $k$-Schur positive. This finishes the proof for the case when $r=0$. 

We next consider the case when $r>0$.  If $c=(0,\ldots, 0)$ is the zero vector, then the  coefficient $F^{[r]}_{\gamma,c}$ is zero, and thus we are done. 
It remains to verify the case when $c$ is not the zero  vector.

Suppose first that $\gamma$ is not a partition and choose $b$ as in the above Claim.  If $c_{b-1}=c_b$, then $Q_c$ is invariant under $\tau_{b-1}$.  Using Proposition \ref{prop:straight}, we may apply induction as in the $r=0$ case to conclude that $F^{[r]}_{\gamma,c}$  is $k$-Schur positive.
Next, consider the case when $c_b>c_{b-1}$.   Note that $c-\epsilon_b$ is again nonnegative and weakly increasing. Note also that by (ii) in the Claim, the path $\mathrm{Path}_\gamma(b)$ is safe.  So, the   identity
\[
 Q_c=Q_{c-\epsilon_b}+u z_b^{-1}Q_c
\]
together with Lemma \ref{lem:path} leads to the following  recursion
\begin{equation}\label{eq:recursion}
 F^{[r]}_{\gamma,c}=
 F^{[r]}_{\gamma,c-\epsilon_b}
 +\sum_{v\in\mathrm{Path}_\gamma(b)}
 F^{[r-1]}_{\gamma-\epsilon_v,c}.
\end{equation}
The first term lowers $|c|$, and every term in the sum lowers $r$.  Applying  induction to every summand yields that $ F^{[r]}_{\gamma,c}$ is $k$-Schur positive.

We finally consider the case when $\gamma$ is  a partition.
In this case, choose any $b\in [\ell]$ such that
\[
 c_b>c_{b-1},
\]
where we set $c_0:=0$.
Note that every path in the bounce graph of a partition is safe. Therefore, we may still invoke Lemma \ref{lem:path} and apply induction as above to conclude that $ F^{[r]}_{\gamma,c}$ is $k$-Schur positive. It should be pointed out that vector $\gamma-\epsilon_v$ in \eqref{eq:recursion}
may no longer be a partition; however, this does not affect our induction since we work in $\widetilde{\mathrm{Par}}_\ell^k$. 
This eventually completes the proof.  
\end{proof}

\section{Concluding remarks}\label{sec-f-9}

As mentioned in Remark \ref{rek-0-k}, 
Blasiak,   Morse  and Seelinger    \cite[Conjecture 2.19]{BMS} proposed   an extension of Theorem \ref{thm:LSS}. In this section,  we shall construct  a counterexample to this  conjectured extension.

For a root ideal \(\Psi \subseteq \Delta_\ell^+\) and 
$\eta\in\mathbb{Z}^\ell$, let
\[
\mathrm{maxband}(\Psi,\eta)
 =\max_{1\le i\le\ell}\bigl(\eta_i+\mathrm{nr}(\Psi)_i\bigr),
\]
where 
\[
 \mathrm{nr}(\Psi)_i
 =\#\{i+1\leq j\leq \ell\colon (i,j)\notin\Psi\}\]
Moreover, set 
\[
 \mathrm{RC}(\Psi)
 =\Psi\setminus\{\text{all removable roots of }\Psi\}.
\]
Let $\mathrm{RC}^0(\Psi)=\Psi$ and  for $a\geq 1$, recursively define 
\[\mathrm{RC}^a(\Psi)=\mathrm{RC}(\mathrm{RC}^{a-1}(\Psi)).\]

\begin{conjecture}[\text{\cite[Conjecture 2.19(2.17)]{BMS}}]\label{conj:BMS}
Let  \(\Psi \subseteq \Delta_\ell^+\) be a root ideal,  and  let \(\lambda \in \mathrm{Par}_\ell^k\) be   such that \(\operatorname{maxband}(\Psi, \lambda) \leq k\).
For 
\(a \geq 0\),
\begin{equation*}
    K(\Psi; \mathrm{RC}^a(\Psi); \lambda) \in\sum_{\nu\in\Par^k_\ell}\mathbb{Z}_{\ge0}\cdot s^{(k)}_\nu.
\end{equation*}

\end{conjecture}

Note that for  $\lambda\in\Par^k_\ell$ and    $a\ge0$, we have
\[\mathrm{RC}^a(\Delta^k(\lambda))=\Delta^{k+a}(\lambda),\]
which can be readily checked by using Remark \ref{rek-0-8}. So when we take  $\Psi=\Delta^{k}(\lambda)$ and $a=1$, Conjecture \ref{conj:BMS} reduces to  Theorem \ref{thm:LSS}.

We next give an example refuting Conjecture \ref{conj:BMS}. 
Consider 
\[
 \ell=8,
 \qquad
 \lambda=(1^8),
 \qquad
 k=3,
 \qquad a=1.
\]
Let  $\Psi\subseteq \Delta_8^+$ and  $\mathrm{RC}(\Psi)$ be the root ideals displayed in   Figure \ref{fig:Psi-RC-counterexample}, where the removable roots of $\Psi$ are colored red.    
\begin{figure}[h t]
\centering

\[
\begin{array}{c@{\qquad\qquad}c}
\begin{array}{c}
\BPD[5mm]{%
  \O\O
  \RemCell{1,3}
  \PsiCell{1,4}
  \PsiCell{1,5}
  \PsiCell{1,6}
  \PsiCell{1,7}
  \PsiCell{1,8}
  \\
  \O\O\O\O
  \RemCell{2,5}
  \PsiCell{2,6}
  \PsiCell{2,7}
  \PsiCell{2,8}
  \\
  \O\O\O\O\O
  \RemCell{3,6}
  \PsiCell{3,7}
  \PsiCell{3,8}
  \\
  \O\O\O\O\O\O
  \PsiCell{4,7}
  \PsiCell{4,8}
  \\
  \O\O\O\O\O\O
  \RemCell{5,7}
  \PsiCell{5,8}
  \\
  \O\O\O\O\O\O\O\O
  \\
  \O\O\O\O\O\O\O\O
  \\
  \O\O\O\O\O\O\O\O
}
\end{array}
&
\begin{array}{c}
\BPD[5mm]{%
  \O\O\O
  \PsiCell{1,4}
  \PsiCell{1,5}
  \PsiCell{1,6}
  \PsiCell{1,7}
  \PsiCell{1,8}
  \\
  \O\O\O\O\O
  \PsiCell{2,6}
  \PsiCell{2,7}
  \PsiCell{2,8}
  \\
  \O\O\O\O\O\O
  \PsiCell{3,7}
  \PsiCell{3,8}
  \\
  \O\O\O\O\O\O
  \PsiCell{4,7}
  \PsiCell{4,8}
  \\
  \O\O\O\O\O\O\O
  \PsiCell{5,8}
  \\
  \O\O\O\O\O\O\O\O
  \\
  \O\O\O\O\O\O\O\O
  \\
  \O\O\O\O\O\O\O\O
}
\end{array}
\end{array}
\]

\caption{
The root ideals $\Psi$ (left) and $\operatorname{RC}(\Psi)$ (right).
}
\label{fig:Psi-RC-counterexample}
\end{figure}

By definition, we see that
\[
 (\mathrm{nr}(\Psi)_1,\ldots, \mathrm{nr}(\Psi)_8)
 =(1,2,2,2,1,2,1,0)\]
 and  
 \[
 (\mathrm{nr}(\Psi)_1+\lambda_1,\ldots, \mathrm{nr}(\Psi)_8+\lambda_8)
 =(2,3,3,3,2,3,2,1).\]
So
\[
 \mathrm{maxband}(\Psi,1^8)=3,
\]
which satisfies the condition required in  Conjecture \ref{conj:BMS}.
By direct computation,   
\begin{align*}
K\bigl(\Psi;\operatorname{RC}(\Psi);1^8\bigr)
={}&
-s_{(2)}^{(3)}
-s_{(1,1)}^{(3)}
+2s_{(3)}^{(3)}
+s_{(2,1)}^{(3)}
-s_{(1,1,1)}^{(3)}
+7s_{(2,2)}^{(3)}
+9s_{(2,1,1)}^{(3)}
\\
&+6s_{(1^4)}^{(3)}
+12s_{(2,2,1)}^{(3)}
+16s_{(2,1^3)}^{(3)}
+15s_{(1^5)}^{(3)}
+4s_{(2,2,1,1)}^{(3)}
\\
&+14s_{(2,1^4)}^{(3)}
+14s_{(1^6)}^{(3)}
+4s_{(2,2,1^3)}^{(3)}
+6s_{(1^7)}^{(3)}
+s_{(1^8)}^{(3)},
\end{align*}
which is not $k$-Schur positive.


\begin{thebibliography}{99}


\bibitem{BMPS}
J.~Blasiak, J.~Morse, A.~Pun, and D.~Summers,
Catalan functions and $k$-Schur positivity,
J. Amer. Math. Soc. \textbf{32} (2019), 921--963.

\bibitem{BMS}
J.~Blasiak, J.~Morse, and G.~H. Seelinger,
$K$-theoretic Catalan functions,
Adv. Math. \textbf{404} (2022), Paper No.~108421.

\bibitem{Chen}
L.-C.~Chen,
Skew-linked partitions and a representation-theoretic model for
$k$-Schur functions,
Ph.D. thesis, University of California, Berkeley, 2010.

 

\bibitem{GKL}
P.L. Guo, M. Kang and  J. Liu, 
Hybrid Grothendieck polynomials, arXiv:2505.19072v1, 2025.


\bibitem{Lam}
T.~Lam,
Schubert polynomials for the affine Grassmannian,
J. Amer. Math. Soc. \textbf{21} (2008), 259--281.



\bibitem{LLMS1}
T.~Lam, L.~Lapointe, J.~Morse, and M.~Shimozono,
Affine insertion and Pieri rules for the affine Grassmannian,
Mem. Amer. Math. Soc. \textbf{208} (2010), no.~977, xii+82.


\bibitem{LP}
T. Lam and  P. Pylyavskyy, Combinatorial Hopf algebras and K-homology of Grassmannians, Int. Math. Res. Not. IMRN(2007), no. 24, Art. ID rnm125, 48 pp.

\bibitem{Survey}
T.~Lam, L.~Lapointe, J.~Morse, A.~Schilling, M.~Shimozono, and M.~Zabrocki,
$k$-Schur functions and affine Schubert calculus,
Fields Institute Monographs, vol.~33, Springer, New York, 2014.

\bibitem{LLM}
L.~Lapointe, A.~Lascoux, and J.~Morse,
Tableau atoms and a new Macdonald positivity conjecture,
Duke Math. J. \textbf{116} (2003), no.~1, 103--146.

\bibitem{LM}
L.~Lapointe and J.~Morse,
Tableaux on $(k+1)$-cores, reduced words for affine permutations, and
$k$-Schur expansions,
J. Combin. Theory Ser. A \textbf{112} (2005), no.~1, 44--81.


\bibitem{LM1}
L.~Lapointe and J.~Morse,
A $k$-tableau characterization of $k$-Schur functions,
Adv. Math. \textbf{213} (2007), no.~1, 183--204.



\bibitem{LSS}
T.~Lam, A.~Schilling, and M.~Shimozono,
$K$-theory Schubert calculus of the affine Grassmannian,
Compos. Math. \textbf{146} (2010), 811--852.

\bibitem{Morse}
J.~Morse,
Combinatorics of the $K$-theory of affine Grassmannians,
Adv. Math. \textbf{229} (2012), no.~5, 2950--2984.

\bibitem{Panyushev}
D. Panyushev,
Generalised Kostka-Foulkes polynomials and cohomology of line bundles on homogeneous vector bundles.
Selecta Math. (N.S.) 16 (2010), no. 2, 315–342.




\end{thebibliography}
\end{document}